\documentclass[a4paper,reqno,11pt]{amsart}
\usepackage{amssymb}
\usepackage{amstext}
\usepackage{amsmath}
\usepackage{amscd}
\usepackage{amsthm}
\usepackage{amsfonts}
\usepackage{enumerate}
\usepackage{graphicx}
\usepackage{latexsym}
\usepackage{mathrsfs}
\input xy
\xyoption{all}
\usepackage{lscape}
\usepackage{pstricks}
\usepackage{comment}
\usepackage{url}

\newtheorem{theorem}{Theorem}[section]
\newtheorem{corollary}[theorem]{Corollary}
\newtheorem{lemma}[theorem]{Lemma}
\newtheorem{definition-theorem}[theorem]{Definition-Theorem}
\newtheorem{proposition}[theorem]{Proposition}
\newtheorem{mainthm}{Theorem}

\theoremstyle{definition}
\newtheorem{definition}[theorem]{Definition}
\newtheorem{remark}[theorem]{Remark}
\newtheorem{example}[theorem]{Example}

\numberwithin{equation}{theorem}

\newcommand{\add}{\mathsf{add}\hspace{.01in}}
\newcommand{\ind}{\mathsf{ind}\hspace{.01in}}

\renewcommand{\mod}{\mathsf{mod}\hspace{.01in}}

\newcommand{\proj}{\mathsf{proj}\hspace{.01in}}

\newcommand{\End}{\operatorname{End}\nolimits}

\newcommand{\gl}{\operatorname{gldim}\nolimits}

\newcommand{\Hom}{\operatorname{Hom}\nolimits}

\newcommand{\im}{\operatorname{Im}\nolimits}

\newcommand{\op}{\operatorname{op}\nolimits}
\newcommand{\pd}{\operatorname{pd}\nolimits}

\newcommand{\soc}{\operatorname{soc}\nolimits}
\renewcommand{\top}{\operatorname{top}\nolimits}

\begin{document}

\title[On an upper bound for the global dimension of ADR algebras]{On an upper bound for the global dimension of \\ Auslander--Dlab--Ringel algebras}

\author[Mayu Tsukamoto]{Mayu Tsukamoto}

\address{Department of Mathematics, Graduate School of Science, Osaka City University, 3-3-138 Sugimoto, Sumiyoshi-ku, Osaka 558-8585, Japan}
\curraddr{Graduate school of Sciences and Technology for Innovation, Yamaguchi University, 1677-1 Yoshida, Yamaguchi 753-8511, Japan}
\email{tsukamot@yamaguchi-u.ac.jp}




\begin{abstract}
Lin and Xi introduced Auslander--Dlab--Ringel (ADR) algebras of semilocal modules as a generalization of original ADR algebras and showed that they are quasi-hereditary.
In this paper, we prove that such algebras are always left-strongly quasi-hereditary.
As an application, we give a better upper bound for global dimension of ADR algebras of semilocal modules. 
Moreover, we describe characterizations of original ADR algebras to be strongly quasi-hereditary.
\end{abstract}

\maketitle

\section*{Introduction}
Quasi-hereditary algebras were introduced by Cline, Parshall and Scott to study highest weight categories which arise in the representation theory of semisimple complex Lie algebras and algebraic groups \cite{{CPS}, {S}}. 
Dlab and Ringel intensely studied quasi-hereditary algebras from the viewpoint of the representation theory of artin algebras \cite{{DR2}, {DR}, {DR6}}. 

Motivated by Iyama's finiteness theorem, Ringel introduced the notion of left-strongly quasi-hereditary algebras in terms of highest weight categories \cite{R}. 
One of the advantages of left-strongly quasi-hereditary algebras is that they have better upper bound for global dimension than that of general quasi-hereditary algebras. 
Moreover, Ringel studied a special class of left-strongly quasi-hereditary algebras called strongly quasi-hereditary algebras. 

Let $A$ be an artin algebra with Loewy length $m$. 
In \cite{A}, Auslander studied the endomorphism algebra $B:=\End_{A}(\bigoplus_{j=1}^{m}A/J(A)^{j})$ and proved that $B$ has finite global dimension.
Furthermore, Dlab and Ringel showed that $B$ is a quasi-hereditary algebra \cite{DR3}. 
Hence $B$ is called an Auslander--Dlab--Ringel (ADR) algebra.
Recently, Conde gave a left-strongly quasi-hereditary structure on ADR algebras \cite{C}. 
Moreover, ADR algebras were studied in \cite{{C2}, {CEr}} and appeared in \cite{{Co}, {KK}}. 

In this paper, we study ADR algebras of semilocal modules introduced by Lin and Xi \cite{LX}.
Recall that a module $M$ is called semilocal if $M$ is a direct summand of modules which have a simple top.
Since any artin algebra is a semilocal module, the ADR algebras of semilocal modules are a generalization of the original ADR algebras. 
In \cite{LX}, they proved that ADR algebras of semilocal modules are quasi-hereditary.
We refine this result in Section \ref{ADR}.

\begin{mainthm} [Theorem \ref{thm1}] \label{thma}
The Auslander--Dlab--Ringel algebra of any semilocal module is left-strongly quasi-hereditary.
\end{mainthm}

As an application, we give a tightly upper bound for global dimension of an ADR algebra (see Corollary \ref{cor}).

In Section \ref{SQHADR}, we study a connection between ADR algebras and strongly quasi-hereditary algebras.
An ADR algebra is a left-strongly quasi-hereditary algebra but not necessarily strongly quasi-hereditary.
We give characterizations of original ADR algebras to be strongly quasi-hereditary. 

\begin{mainthm} [Theorem \ref{thm2}] \label{thmb}
Let $A$ be an artin algebra with Loewy length $m \geq 2$ and $J$ the Jacobson radical of $A$.
Let $B:=\End_A(\bigoplus_{j=1}^{m} A/J^j)$ be the ADR algebra of $A$.
Then the following statements are equivalent.
\begin{itemize}
\item[{\rm (i)}] $B$ is a strongly quasi-hereditary algebra. 

\item[{\rm (ii)}] $\gl B =2$.

\item[{\rm (iii)}] $J \in \add (\bigoplus_{j=1}^{m} A/J^j)$. 
\end{itemize}
\end{mainthm}

It is known that if $B$ is strongly quasi-hereditary, then the global dimension of $B$ is at most two \cite[Proposition A.2]{R}.
We note that algebras with global dimension at most two are not necessarily strongly quasi-hereditary.
However, for original ADR algebras, the converse is also true.

\section{Preliminaries}
\medskip
\subsection*{Notation}
Let $A$ be an artin algebra, $J(A)$ the Jacobson radical of $A$ and $\mathrm{D}$ the Matlis dual. 
We denote by $\gl A$ the global dimension of $A$.
We fix a complete set of representatives of isomorphism classes of simple $A$-modules $\{ S(i) \; | \; i \in I \}$.
We denote by $P(i)$ the projective cover of $S(i)$ and $E(i)$ the injective hull of $S(i)$ for any $i \in I$. 

We write $\mod A$ for the category of finitely generated right $A$-modules and $\proj A$ for the full subcategory of $\mod A$ consisting of finitely generated projective $A$-modules. 
For $M \in \mod A$, we denote by $\add M$ the full subcategory of $\mod A$ whose objects are direct summands of finite direct sums of $M$.

The composition of two maps $f:X \to Y$ and $g: Y \to Z$ is denoted by $g \circ f$.
For a quiver $Q$, we denote by $\alpha \beta$ the composition of two arrows $\alpha: x \to y$ and $\beta: y \to z$ in $Q$.

We denote by $K$ an algebraically closed field.
\medskip

In this section, we quickly review a relationship between strongly quasi-hereditary algebras and rejective chains.
For more detail, we refer to \cite{{I2}, {T}}.

We start this section with recalling the definition of left-strongly quasi-hereditary algebras.
Let $\leq$ be a partial order on the index set $I$ of simple $A$-modules.
For each $i \in I$, we denote by $\nabla(i)$ the maximal submodule of $E(i)$ whose composition factors have the form $S(j)$ for some $j \leq i$.
The module $\nabla(i)$ is called the {\it costandard module} corresponding to $i$.
Let $\nabla:= \{ \nabla(i) \; | \; i \in I \}$ be the set of costandard modules.
We denote by $\mathcal{F}(\nabla)$ the full subcategory of $\mod A$ whose objects are the modules which have a $\nabla$-filtration, that is, $M \in \mathcal{F}(\nabla)$ if and only if there exists a chain of submodules 
\[
M=M_0 \supseteq M_1 \supseteq \cdots \supseteq M_l=0
\]
such that $M_i/M_{i+1}$ is isomorphic to a module in $\nabla$. 
For $M \in \mathcal{F}(\nabla)$, we denote by $(M: \nabla(i))$ the filtration multiplicity of $\nabla(i)$, which dose not depend on the choice of $\nabla$-filtrations.

\begin{definition}[{\cite[\S 4]{R}}]
Let $A$ be an artin algebra and $\leq$ a partial order on $I$.
\begin{itemize}
\item[{\rm (1)}] A pair $(A, \leq)$ (or simply $A$) is called {\it left-strongly quasi-hereditary} if there exists a short exact sequence
\begin{equation*}
0 \to \nabla(i) \to E(i) \to E(i) /\nabla(i) \to 0  
\end{equation*}
for any $i \in I$ with the following properties:
\begin{enumerate}
\item [{\rm (a)}] $E(i)/\nabla(i) \in \mathcal{F}(\nabla)$ for any $i \in I$;  
\item [{\rm (b)}] if $(E(i)/\nabla(i) :\nabla(j)) \not= 0$, then we have $i < j$;
\item [{\rm (c)}] $E(i)/\nabla(i)$ is an injective $A$-module, or equivalently, $\nabla(i)$ has injective dimension at most one.
\end{enumerate}
\item[{\rm (2)}] We say that a pair $(A, \leq)$ (or simply $A$) is {\it right-strongly quasi-hereditary} if $(A^{\op}, \leq)$ is left-strongly quasi-hereditary. 
\item[{\rm (3)}] We say that a pair $(A, \leq)$ (or simply $A$) is {\it strongly quasi-hereditary} if $(A, \leq)$ is left-strongly quasi-hereditary and right-strongly quasi-hereditary. 
\end{itemize}
\end{definition}

By definition, strongly quasi-hereditary algebras are left-strongly quasi-hereditary algebras. 
Since a pair $(A, \leq)$ satisfying the conditions (a) and (b) is a quasi-hereditary algebra, left-strongly quasi-hereditary algebras are quasi-hereditary.   

Left-strongly (resp.\ right-strongly) quasi-hereditary algebras are characterized by total left (resp.\ right) rejective chains, which are chains of certain left (resp.\ right) rejective subcategories.
We recall the notion of left (resp.\ right) rejective subcategories. 
Let $\mathcal{C}$ be an additive category, and put $\mathcal{C}(X, Y):=\Hom_{\mathcal{C}}(X,Y)$. 
In this section, {\it we assume that any subcategory is full and closed under isomorphisms, direct sums and direct summands.}

\begin{definition} [{\cite[2.1(1)]{I}}] \label{rejsub}
Let $\mathcal{C}$ be an additive category.
A subcategory $\mathcal{C}'$ of $\mathcal{C}$ is called 
\begin{itemize}
\item[{\rm (1)}] a \emph{left $($resp.\ right$)$ rejective subcategory} of $\mathcal{C}$ if, for any $X\in\mathcal{C}$, there exists an epic left $($resp.\ monic right$)$ $\mathcal{C}'$-approximation $f^X \in \mathcal{C}\left(X,Y\right)$ $($resp.\ $f_X \in \mathcal{C}\left(Y,X\right))$ of $X$, 
\item[{\rm (2)}] a \emph{rejective subcategory} of $\mathcal{C}$ if $\mathcal{C}'$ is a left and right rejective subcategory of $\mathcal{C}$.
\end{itemize}
\end{definition}

To define a total left (resp.\ right) rejective chain, we need the notion of cosemisimple subcategories. 
Let $\mathcal{J}_{\mathcal{C}}$ be the Jacobson radical of $\mathcal{C}$.
For a subcategory $\mathcal{C}'$ of $\mathcal{C}$, we denote by $[\mathcal{C}']$ the ideal of $\mathcal{C}$ consisting of morphisms which factor through some object of $\mathcal{C}'$, and by $\mathcal{C}/[\mathcal{C}']$ the factor category (\emph{i.e.}, $\mathit{ob}(\mathcal{C}/[\mathcal{C}']):=\mathit{ob}(\mathcal{C})$ and $(\mathcal{C}/[\mathcal{C}'])(X,Y):= \mathcal{C}(X,Y)/[\mathcal{C}'](X,Y)$ for any $X, Y \in \mathcal{C}$). 
Recall that an additive category $\mathcal{C}$ is called a {\it Krull--Schmidt} category if any object of $\mathcal{C}$ is isomorphic to a finite direct sum of objects whose endomorphism rings are local. 
We denote by $\ind \mathcal{C}$ the set of isoclasses of indecomposable objects in $\mathcal{C}$. 

\begin{definition}
Let $\mathcal{C}$ be a Krull--Schmidt category.
A subcategory $\mathcal{C}'$ of $\mathcal{C}$ is called {\it cosemisimple} in $\mathcal{C}$ if $\mathcal{J}_{\mathcal{C}/[\mathcal{C'}]}=0$ holds.
\end{definition}

We give a characterization of cosemisimple left rejective subcategories.

\begin{proposition} [{\cite[1.5.1]{I2}}] \label{crrs}
Let $\mathcal{C}$ be a Krull--Schmidt category and let $\mathcal{C}'$ be a subcategory of $\mathcal{C}$. 
Then $\mathcal{C}'$ is a cosemisimple left $($resp.\ right$)$ rejective subcategory of $\mathcal{C}$ if and only if, for any $X \in \ind \mathcal{C} \setminus \ind \mathcal{C}'$,
there exists a morphism $\varphi: X \to Y$ $($resp.\ $\varphi : Y \to X)$ such that $Y \in \mathcal{C}'$ and $\mathcal{C}(Y, -) \xrightarrow{-\circ \varphi} \mathcal{J}_{\mathcal{C}}(X, -)$ $($resp.\ $\mathcal{C}(-, Y) \xrightarrow{\varphi \circ -} \mathcal{J}_{\mathcal{C}}(-, X))$ is an isomorphism on $\mathcal{C}$. 
\end{proposition}

Now, we introduce the following key notion in this paper.

\begin{definition}[{\cite[2.1(2)]{I}}] \label{rejch}
Let $\mathcal{C}$ be a Krull--Schmidt category.
A chain 
\begin{equation*} 
\mathcal{C}= \mathcal{C}_0 \supset \mathcal{C}_1 \supset \cdots \supset \mathcal{C}_n =0
\end{equation*}
of subcategories of $\mathcal{C}$ is called 
\begin{itemize}
\item[{\rm (1)}] a \emph{rejective chain} if $\mathcal{C}_i$ is a cosemisimple rejective subcategory of $\mathcal{C}_{i-1}$ for $1 \leq i \leq n$,

\item[{\rm (2)}] a \emph{total left $(${\rm resp.}\ right$)$ rejective chain} if the following conditions hold for $1 \leq i \leq n$:
\begin{enumerate}
\item[(a)] $\mathcal{C}_i$ is a left (resp.\ right) rejective subcategory of $\mathcal{C}$;
\item[(b)] $\mathcal{C}_{i}$ is a cosemisimple subcategory of $\mathcal{C}_{i-1}$.
\end{enumerate}
\end{itemize}
\end{definition}

The following proposition gives a connection between left-strongly quasi-hereditary algebras and total left rejective chains.

\begin{proposition} [{\cite[Theorem 3.22]{T}}] \label{thm0}
Let $A$ be an artin algebra.
Let $M$ be a right $A$-module and $B:= \End_A(M)$. 
Then the following conditions are equivalent. 
\begin{itemize}
\item[{\rm (i)}] $B$ is a left-strongly $($resp.\ right-strongly$)$ quasi-hereditary algebra.

\item[{\rm (ii)}] $\proj B$ has a total left $($resp.\ right$)$ rejective chain. 

\item[{\rm (iii)}] $\add M$ has a total left $($resp.\ right$)$ rejective chain.
\end{itemize}

In particular, $B$ is strongly quasi-hereditary if and only if $\add M$ has a rejective chain.
\end{proposition}

We end this section with recalling a special total left rejective chain, which plays an important role in this paper.
 
\begin{definition}[{\cite[Definition 2.2]{I2}}]
Let $A$ be an artin algebra and $\mathcal{C}$ a subcategory of $\mod A$.
A chain 
\begin{equation*}
\mathcal{C}= \mathcal{C}_0 \supset \mathcal{C}_1 \supset \cdots \supset \mathcal{C}_n =0
\end{equation*}
of subcategories of $\mathcal{C}$ is called an \emph{$A$-total left $(${\rm resp.}\ right$)$ rejective chain of length $n$} if the following conditions hold for $1 \leq i \leq n$:
\begin{enumerate}
\item[(a)] for any $X \in \mathcal{C}_{i-1}$, there exists an epic $(${\rm resp.}\ monic$)$ in $\mod A$ left $(${\rm resp.}\ right$)$ $\mathcal{C}_i$-approximation of $X$;
\item[(b)] $\mathcal{C}_{i}$ is a cosemisimple subcategory of $\mathcal{C}_{i-1}$.
\end{enumerate}
\end{definition}

All $A$-total left rejective chains of $\mathcal{C}$ are total left rejective chains. 
Moreover, If $\mathrm{D}A \in \mathcal{C}$, then the converse also holds.

We can give an upper bound for global dimension by using $A$-total left rejective chains.

\begin{proposition}[{\cite[Theorem 2.2.2]{I2}}] \label{iygl} 
Let $A$ be an artin algebra and $M$ a right $A$-module. 
If $\add M$ has an $A$-total left $($resp.\ right$)$ rejective chain of length $n>0$, then $\gl \End_A(M) \leq n$ holds.
\end{proposition}
 
\section{ADR algebras of semilocal modules} \label{ADR}
The aim of this section is to show Theorem \ref{thma}.
First, we recall the definition of semilocal modules.

\begin{definition}
Let $M$ be an $A$-module. 

\begin{itemize}
\item[(1)] $M$ is called a {\it local} module if $\top M$ is isomorphic to a simple $A$-module.
\item[(2)] $M$ is called a {\it semilocal} module if $M$ is a direct sum of local modules.
\end{itemize}
\end{definition}
 
Clearly, any local module is indecomposable and any projective module is semilocal.
 
Throughout this section, suppose that $\displaystyle M$ is a semilocal module with Loewy length $\ell \ell (M)=m$.
We denote by $\widetilde{M}$ the basic module of $\oplus_{i=1}^{m} M/MJ(A)^i$ and call $\End_A(\widetilde{M})$ the {\it Auslander--Dlab--Ringel algebra} (ADR algebra) of $M$. 
Note that $\End_A(\widetilde{A})$ is an ADR algebra in the sense of \cite{C}. 

Lin and Xi showed that the ADR algebras of semilocal modules are quasi-hereditary (see \cite[Theorem]{LX}).
In this section, we refine this result.

\begin{theorem}\label{thm1}
The ADR algebra of any semilocal module is left-strongly quasi-hereditary.
\end{theorem}

Observe that Theorem \ref{thm1} gives a better upper bound for global dimension of ADR algebras (see Remark \ref{rem}).

In the following, we give a proof of Theorem \ref{thm1}.
Let $\mathsf{F}$ be the set of pairwise non-isomorphic indecomposable direct summands of $\widetilde{M}$ and $\mathsf{F}_i$ the subset of $\mathsf{F}$ consisting of all modules with Loewy length $m-i$.
We denote by $\mathsf{F}_{i, 1}$ the subset of $\mathsf{F}_i$ consisting of all modules $X$ which do not have a surjective map in $\mathcal{J}_{\mod A}(X,N)$ for all modules $N$ in $\mathsf{F}_i$. For any integer $j>1$, we inductively define the subsets $\mathsf{F}_{i,j}$ of $\mathsf{F}_{i}$ as follows: $\mathsf{F}_{i, j}$ consists of all modules $X\in\mathsf{F}_{i}\setminus \bigcup_{1 \leq k \leq j-1}\mathsf{F}_{i,k}$ which do not have a surjective map in $\mathcal{J}_{\mod A}(X,N)$ for all modules $N\in\mathsf{F}_{i}\setminus \bigcup_{1 \leq k \leq j-1}\mathsf{F}_{i,k}$. 
We set $n_{i}:=\min\{j\mid \mathsf{F}_{i}=\bigcup_{1 \leq k \leq j}\mathsf{F}_{i,k}\}$ and $n_{M}:=\sum_{i=0}^{m-1} n_i$. 
For $0 \leq i \leq m-1$ and $1 \leq j \leq n_i$, we set
\begin{align}
\mathsf{F}_{>(i,j)}&:=\mathsf{F}\setminus ((\cup_{-1\le k \le i-1}\mathsf{F}_{k})\cup (\cup_{1\le l \le j}\mathsf{F}_{i,l})),\notag \\
\mathcal{C}_{i,j} &:= \add \bigoplus_{N \in \mathsf{F}_{> (i,j)}} N, \notag
\end{align}
where $\mathsf{F}_{-1}:=\emptyset$. 

Now, we display an example to explain how the subsets $\mathsf{F}_{i, j}$ are given.

\begin{example} \label{eg}
Let $A$ be the $K$-algebra defined by the quiver
\[
\xymatrix@=15pt{ 1 \ar[r] & 2 \ar[r] \ar [d] & 3 \\
 &  4
 }
 \]
and $M:= P(1) \oplus P(1)/S(3) \oplus P(1)/S(4) \oplus P(2)/S(3)$.
We can easily check that $M$ is a semilocal module.
The ADR algebra $B$ of $M$ is given by the quiver
\[
\xymatrix@=15pt{ P(1)/S(4) \ar[r]^{a} & P(1) & P(1)/S(3) \ar[l]_{b} \ar[rd]^{c} \\
 & P(1)/ P(1) J(A)^2 \ar[lu]^{d} \ar[ru]_{e} \ar[rd]_{f} & &P(2)/ S(3) \\
 & S(1) \ar[u]^{g} & S(2) \ar [ru]_{h}
 }
 \]
with relations $da-eb, ec-fh$ and $gf$.
Then $\mathsf{F}_{0,1}=\{ P(1)/S(4), P(1)/S(3) \}$, $\mathsf{F}_{0,2}=\{ P(1) \}$, $\mathsf{F}_{1,1}=\{ P(1)/P(1)J(A)^2, P(2)/S(3) \}$, $\mathsf{F}_{2,1} = \{ S(1), S(2) \}$. 
\end{example}

To prove Theorem \ref{thm1}, we first show the following proposition. 

\begin{proposition} \label{prop}
Let $A$ be an artin algebra and $M$ a semilocal $A$-module.
Then $\add \widetilde{M}$ has the following $A$-total left rejective chain with length $n_{M}$.
\begin{equation*}
\add \widetilde{M} =:\mathcal{C}_{0,0}\supset \mathcal{C}_{0,1} \supset \cdots \supset \mathcal{C}_{0,n_0} \supset \mathcal{C}_{1,1} \supset \cdots \supset \mathcal{C}_{m-1,n_{m-1}}=0.
\end{equation*} 
\end{proposition}

To show Proposition \ref{prop}, we need the following lemma.

\begin{lemma} \label{lem2}
For any $M' \in \mathsf{F}_{0,1}$, the canonical surjection $\rho : M' \twoheadrightarrow M'/M' J(A)^{m-1}$ induces an isomorphism
\begin{equation*}
\varphi : \Hom_A (M'/M' J(A)^{m-1}, \widetilde{M}) \xrightarrow{- \circ \rho} \mathcal{J}_{\mod A}(M', \widetilde{M}).
\end{equation*}
\end{lemma}

\begin{proof}
Since $\varphi$ is a well-defined injective map, we show that $\varphi$ is surjective.
Let $N$ be an indecomposable summand of $\widetilde{M}$ with Loewy length $k$ and let $f:M' \to N$ be any morphism in $\mathcal{J}_{\mod A}(M', N)$.
Then we show $f(M' J(A)^{m-1})=0$.

(i) Assume that $\top M' \not \cong \top N$ or $k=m$.
Then we have $\im f\subset NJ(A)$, and hence
\begin{align}
f(M' J(A)^{m-1})=f(M')J(A)^{m-1}\subset (NJ(A))J(A)^{m-1}=0. \notag
\end{align}

(ii) Assume that $\top M' \cong \top N$ and $k<m$.
Since $m-k>0$ holds, we obtain 
\begin{align}
f(M' J(A)^{m-1})=f(M')J(A)^{m-1}\subset NJ(A)^{m-1}= (NJ(A)^{k})J(A)^{m-k-1}=0. \notag
\end{align}

Since $f(M' J(A)^{m-1})=0$ holds, there exists $g: M' /M' J(A)^{m-1}\to N$ such that $f=g \circ \rho$.
\begin{align}
\xymatrix{
0\ar[r]&M' J(A)^{m-1}\ar[r]\ar[rd]_{0}&M' \ar[r]^{\rho\hspace{10mm}}\ar[d]^{f}&M' /M' J(A)^{m-1}\ar[r]\ar@{-->}[dl]^{\exists{g}}&0\\
&&N&&
}\notag
\end{align}
Hence the assertion follows.
\end{proof}

Now, we are ready to prove Proposition \ref{prop}.

\begin{proof}[Proof of Proposition \ref{prop}]
We show by induction on $n_{M}$.
If $n_{M}=1$, then this is clear.
Assume that $n_{M} >1$. 
By Proposition \ref{crrs} and Lemma \ref{lem2}, $\mathcal{C}_{0,1}$ is a cosemisimple left rejective subcategory of $\add \widetilde{M}$.
Since $N:=M/(\oplus_{X \in \mathsf{F}_{0,1}} X) \oplus (\oplus_{X \in \mathsf{F}_{0,1}} X/ X J(A)^{m-1})$ is a semilocal module satisfying $\widetilde{N}=\widetilde{M}/\oplus_{X\in\mathsf{F}_{0,1}}X$ and $n_{N}<n_{M}$, we obtain that 
\begin{align}
\add \widetilde{N}= \mathcal{C}_{0,1} \supset \cdots \supset \mathcal{C}_{0,n_{0}}\supset \mathcal{C}_{1,1} \supset \cdots \supset \mathcal{C}_{m-1, n_{m-1}} =0\notag
\end{align}
is an $A$-total left rejective chain by induction hypothesis.
By composing $\mathcal{C}_{0,0} \supset \mathcal{C}_{0,1}$ and it, we have the desired $A$-total left rejective chain.
\end{proof}

\begin{proof}[Proof of Theorem \ref{thm1}]
By Proposition \ref{thm0}, it is enough to show that $\add \widetilde{M}$ has a total left rejective chain.
Hence the assertion follows from Proposition \ref{prop}.
\end{proof}

We give some remark on partial orders for left-strongly quasi-hereditary algebras
\begin{remark}
We define two partial orders on the isomorphism classes of simple $B$-modules.
One is $\{ \mathsf{F}_{0,1} < \cdots < \mathsf{F}_{0,n_0} < \mathsf{F}_{1,1} < \cdots < \mathsf{F}_{m-1, n_{m-1}} \}$, called the {\it ADR order}. 
Another one is $\{ \mathsf{F}_{0} < \mathsf{F}_{1} < \cdots < \mathsf{F}_{m-1} \}$, called the {\it length order}.

By Proposition \ref{prop}, ADR algebras of semilocal modules are left-strongly quasi-hereditary with respect to the ADR order.
On the other hand, Conde shows that original ADR algebras are left-strongly quasi-hereditary with respect to the length order \cite{C}.
Since, for an original ADR algebra, the length order coincides with the ADR order, we can recover Conde's result.
However, the ADR algebra of a semilocal module is not necessarily left-strongly quasi-hereditary with respect to the length order, as shown by the following example. 
\end{remark}

\begin{example}
Let $A$ and $M$ be in Example \ref{eg}. 
Then we can check that the ADR algebra $B$ of $M$ is left-strongly quasi-hereditary with respect to the ADR order 
\begin{equation*}
\{ \mathsf{F}_{0,1} < \mathsf{F}_{0,2} < \mathsf{F}_{1,1} < \mathsf{F}_{2,1} \}.
\end{equation*}
However, we can also check that $B$ is not left-strongly quasi-hereditary with respect to the length order 
\begin{align}
\{ \{P(1)/S(3), P(1)/S(4), P(1)\} < \{P(1)/P(1)J(A)^2, P(2)/S(3)\} < \{S(1), S(2)\} \}.\notag
\end{align}
\end{example}

As an application, we give an upper bound for global dimension of ADR algebras.

\begin{corollary} \label{cor}
Let $A$ be an artin algebra and $M$ a semilocal $A$-module.
Then 
\begin{align}
\gl \End_{A}(\widetilde{M})\le n_{M}. \notag
\end{align}
\end{corollary}

\begin{proof}
By Proposition \ref{prop}, $\add\widetilde{M}$ has an $A$-total left rejective chain with length $n_{M}$.
Hence the assertion follows from Proposition \ref{iygl}.
\end{proof}

\begin{remark} \label{rem}
In \cite{LX}, they showed that the ADR algebra of a semilocal module $M$ is quasi-hereditary.
This implies $\gl \End_{A}(\widetilde{M})\le 2(n_{M}-1)$ by \cite[Statement 9]{DR}.
By Corollary \ref{cor}, we can obtain a better upper bound for global dimension of ADR algebras. 
This can be seen by the following example.
\end{remark}

The following example tells us that the upper bound for the global dimension in Corollary \ref{cor} is tightly.

Let $n \geq 2$. Let $A$ be the $K$-algebra defined by the quiver
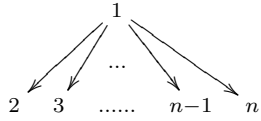
\[
\def\objectstyle{\scriptstyle}
\def\labelstyle{\scriptstyle}
\vcenter{
\hbox{
$
\xymatrix@C=6pt{
& & 1 \ar[lld] \ar[ld] \ar@{}[d]|(.6){\dots} \ar[rd] \ar[rrd]
& & \\
2 & 3 & \dots\dots & n-1 & n
}
$
}
}
\] 
and $M$ a direct sum of all factor modules of $P(1)$. Clearly, $M$ is semilocal and $n_M=n$. 
Let $B$ be its ADR algebra. Then we have
\begin{align*}
\gl B = 
\begin{cases}
n-1 & (n \geq 3)\\
2 & (n=2).
\end{cases}
\end{align*}
Indeed, the assertion for $n=2$ clearly holds.
Assume $n \geq 3$.
It is easy to check that, for $X \in \mathsf{F}_{0,l}$ ($1 \leq l \leq n_0$), 
\begin{align*}
\pd_{B^{\op}} \top (\Hom_A(X, \widetilde{M})) =l. \notag
\end{align*}
Thus we have 
\begin{align*}
\max \{ \pd_{B^{\op}} \top (\Hom_A(X, \widetilde{M})) \mid X \in \mathsf{F} \} = n_0 = n_M -1. \notag
\end{align*}
Hence the assertion for $n \geq 3$ holds.

\section{Strongly quasi-hereditary ADR algebras} \label{SQHADR}

In this section, we prove Theorem \ref{thmb}. 
We keep the notation of the previous section.
Throughout this section, $A$ is an artin algebra with Loewy length $m$ and $B:=\End_{A}(\widetilde{A})$ the ADR algebra of $A$.
Then $n_j =1$ holds for any $0 \leq j \leq m-1$. 
Hence we obtain the following $A$-total left rejective chain by Proposition \ref{prop}.
\begin{equation} \label{rej}
\add \widetilde{A}  \supset \mathcal{C}_{0,1} \supset \mathcal{C}_{1,1} \supset \cdots \supset \mathcal{C}_{m-1,1} = 0.
\end{equation}
Note that if $m=1$, then $B$ is semisimple.
Hence we always assume $m\ge 2$ in the rest of section.

\begin{theorem}\label{thm2}
Let $A$ be an artin algebra with Loewy length $m \geq 2$ and $B$ the ADR algebra of $A$.
Then the following statements are equivalent.
\begin{itemize}
\item[{\rm (i)}] $B$ is a strongly quasi-hereditary algebra. 
\item[{\rm (ii)}] The chain \eqref{rej} is a rejective chain of $\add \widetilde{A}$.
\item[{\rm (iii)}] $\gl B =2$.
\item[{\rm (iv)}] $J(A) \in \add \widetilde{A}$. 
\end{itemize}
\end{theorem}

To prove Theorem \ref{thm2}, we need the following lemma.

\begin{lemma}\label{lem3}
Let $A$ be an artin algebra.
If $P(i)J(A) \in \add \widetilde{A}$ for any $i \in I$, then $P(i)J(A)/P(i)J(A)^j \in \add \widetilde{A}$ for $1 \leq j \leq m$.
\end{lemma}

\begin{proof}
Since $P(i)J(A) \in \add \widetilde{A}$, we have $P(i)J(A) \cong \displaystyle{\bigoplus_{k, l} P(k)/P(k)J(A)^l}$.
For simplicity, we write $P(i)J(A) \cong P(k)/P(k)J(A)^l$.
Then we have $P(i)J(A)/P(i)J(A)^j \cong (P(k)/P(k)J(A)^l)/(P(k)J(A)^j/P(k)J(A)^l) \cong P(k)/P(k)J(A)^j \in \add \widetilde{A}$.
\end{proof}

\begin{proof}[Proof of Theorem \ref{thm2}]
(ii) $\Rightarrow$ (i): The assertion follows from Proposition \ref{thm0}.

(i) $\Rightarrow$ (iii): It follows from \cite[Proposition A.2]{R} that the global dimension of $B$ is at most two.
It is enough to show that there exists a $B$-module such that its projective dimension is two.
Let $S$ be a simple $A$-module.
Then we have the following short exact sequence.
\begin{equation*}
0 \to \mathcal{J}_{\mod A}(\widetilde{A},S) \to \Hom_A (\widetilde{A},S) \to \top \Hom_A(\widetilde{A},S) \to 0. 
\end{equation*}
Assume that $\mathcal{J}_{\mod A}(\widetilde{A},S)$ is a projective right $B$-module.
Then there exists an $A$-module $Y \in \add \widetilde{A}$ such that $\mathcal{J}_{\mod A}(\widetilde{A},S) \cong \Hom_A(\widetilde{A},Y)$.
By $S \in \add \widetilde{A}$, there exists a non-zero morphism $f : Y \to S$ such that $\Hom_A(\widetilde{A},f) : \Hom_A(\widetilde{A}, Y) \to \Hom_A(\widetilde{A}, S)$ is an injective map.
Since the functor $\Hom_{A}(\widetilde{A},-)$ is faithful, $f$ is an injective map.
Hence $f$ is an isomorphism.
This is a contradiction since $\mathcal{J}_{\mod A}(\widetilde{A},S) \cong \Hom_A(\widetilde{A},S)$.
Therefore, we obtain the assertion.

(iii) $\Leftrightarrow$ (iv): This follows from \cite[Proposition 2]{Sm}. 

(iv) $\Rightarrow$ (ii): 
First, we show that $\mathcal{C}_{0,1}$ is a cosemisimple rejective subcategory of $\add \widetilde{A}$.
By Proposition \ref{prop}, it is enough to show that $\mathcal{C}_{0,1}$ is a right rejective subcategory of $\add \widetilde{A}$.
For any $X \in \ind(\add \widetilde{A}) \setminus \ind(\mathcal{C}_{0,1})$, there exists an inclusion map $\varphi : XJ(A) \hookrightarrow X$ with $XJ(A) \in \mathcal{C}_{0,1}$ by the condition (iv).
Since $X$ is a projective $A$-module such that its Loewy length coincides with the Loewy length of $A$, the map $\varphi$ induces an isomorphism
\begin{equation*}
\Hom_A(\widetilde{A}, XJ(A)) \xrightarrow{\varphi \circ -} \mathcal{J}_{\mod A}(\widetilde{A}, X).
\end{equation*}
It follows from Proposition \ref{crrs} that $\mathcal{C}_{0,1}$ is a cosemisimple right rejective subcategory of $\add \widetilde{A}$.
Hence we obtain that $\mathcal{C}_{0,1}$ is a cosemisimple rejective subcategory of $\add \widetilde{A}$.

Next, we prove that $\add \widetilde{A}$ has a rejective chain 
\begin{equation*}
\add \widetilde{A} \supset \mathcal{C}_{0,1} \supset \mathcal{C}_{1,1} \supset \cdots \supset \mathcal{C}_{m-1,1} = 0
\end{equation*}
by induction on $m$. 
If $m=2$, then the assertion holds.
Assume that $m \geq 3$.
Let $X \in \ind(\mathcal{C}_{0,1}) \setminus \ind(\mathcal{C}_{1,1})$.
Then $X=P(i)/P(i)J(A)^{m-1}$ for some $i \in I$ and we have
\begin{equation*}
(P(i)/P(i)J(A)^{m-1})J(A/J^{m-1}(A)) \cong P(i)J(A)/P(i)J(A)^{m-1}.
\end{equation*}
Since $P(i)J(A) \in \add \widetilde{A}$, we obtain $P(i)J(A)/P(i)J(A)^{m-1} \in \mathcal{C}_{0,1}$ by Lemma \ref{lem3}.
By induction hypothesis, $\mathcal{C}_{0,1}$ has the following rejective chain.
\begin{equation*}
\mathcal{C}_{0,1} \supset \mathcal{C}_{1,1} \supset \cdots \supset \mathcal{C}_{m-1,1} = 0. 
\end{equation*}
Composing it with $\add \widetilde{A} \supset \mathcal{C}_{0,1}$, we obtain a rejective chain of $\add \widetilde{A}$. 
\end{proof}

By Theorem \ref{thm2}(i) $\Rightarrow$ (ii), a strongly quasi-hereditary structure of the ADR algebra $B$ can be always realized by the ADR order. However, for a semilocal module, such an assertion does not necessarily hold. In fact, we give an example that the ADR algebra of a semilocal module is strongly quasi-hereditary but not strongly quasi-hereditary with respect to the ADR order.

\begin{example}
Let $A$ be the $K$-algebra defined by the quiver
\[
\xymatrix@=15pt{ 1 \ar@(lu,ld)_{\alpha} \ar[r]^{\beta} & 2
 }
 \]
 with relations $\alpha \beta$ and $\alpha^3$.
 Clearly, $M:= P(1) \oplus P(1)/\soc P(1) \oplus P(2)$ is a semilocal module.
 The ADR algebra $B$ of $M$ is given by the quiver
\[
\xymatrix@=15pt{ &  P(1) \ar@<-1.5ex>[ldd]_{a} \ar[rd]^{b} & \\
& P(1)/P(1)J(A)^2 \ar[u]^{c} \ar[d]_{d} &P(1)/\soc P(1) \ar[l]^{e} \\
P(2) &S(1) \ar[ru]_{f} &
}
 \]
 with relations $eca, fed$ and $cb-df$.
 Then $B$ is not strongly quasi-hereditary with respect to the ADR order $\{ \mathsf{F}_{0,1} < \mathsf{F}_{1,1} < \mathsf{F}_{1,2} < \mathsf{F}_{2,1} \}$, but $B$ is strongly quasi-hereditary with respect to $\{ P(1) < P(1)/P(1)J(A)^2 < P(1)/\soc P(1) < \{P(2), S(1) \} \}$.
\end{example}

\subsection*{Acknowledgment}
The author wishes to express her sincere gratitude to Takahide Adachi and Professor Osamu Iyama. 
The author thanks Teresa Conde and Aaron Chan for informing her about the reference \cite[Proposition 2]{Sm}, which greatly shorten her original proof.

\end{document}